\documentclass{amsart}
\pdfoutput=1

\usepackage{microtype}
\usepackage[mathscr]{eucal}
\usepackage{mathtools}
\usepackage{amsthm}
\usepackage{enumitem}
\usepackage{hyperref}
\usepackage{mleftright}

\newcommand\+{\mkern1.5mu}

\DeclareMathOperator{\Ima}{Im}
\DeclareMathOperator{\coim}{coim}
\DeclareMathOperator{\Rank}{rank}
\DeclareMathOperator{\Span}{span}
\DeclareMathOperator{\Sign}{sgn}

\theoremstyle{plain}
\newtheorem{theorem}{Theorem}[section]

\newtheorem{lemma}[theorem]{Lemma}
\newtheorem{corollary}[theorem]{Corollary}
\theoremstyle{remark}
\newtheorem{remark}[theorem]{Remark}
\theoremstyle{definition}
\newtheorem{problem}[theorem]{Problem}
\newtheorem{definition}[theorem]{Definition}

\begin{document}
	
\title[The Cauchy problem for constant-rank submanifolds]{The geometric Cauchy problem\\for constant-rank submanifolds}

\author{Matteo Raffaelli}
\address{School of Mathematics, Georgia Institute of Technology, Atlanta, Georgia 30332}
\email{raffaelli@math.gatech.edu}
\date{September 9, 2025}

\begin{abstract}
Given a smooth $s$-dimensional submanifold $S$ of $\mathbb{R}^{m+c}$ and a smooth distribution $\mathcal{D}\supset TS$ of rank $m$ along $S$, we study the following geometric Cauchy problem: to find an $m$-dimensional rank-$s$ submanifold $M$ of $\mathbb{R}^{m+c}$ (that is, an $m$-submanifold with constant index of relative nullity $m-s$) such that $M \supset S$ and $TM |_{S} = \mathcal{D}$. In particular, under some reasonable assumption and using a constructive approach, we show that a solution exists and is unique in a neighborhood of $S$.
\end{abstract}
\maketitle

\section{Introduction and main result}

Given a manifold $Q^{q\geq 3}$ and some class $\mathscr{A}$ of $(m<q)$-dimensional submanifolds, the \textit{geometric Cauchy problem for the class $\mathscr{A}$} is to find all members of $\mathscr{A}$ passing from a given $(s<m)$-dimensional submanifold $S$ of $Q$ with prescribed tangent bundle along $S$. This problem, which represents a far-reaching generalization of the classical Bj\"{o}rling problem for minimal surfaces in $\mathbb{R}^{3}$~\cite[section~22.6]{gray2006}, has recently been examined, always with $s=1$, for several combinations of $Q$ and $\mathscr{A}$; see, e.g., \cite{milan2020, bueno2019, brander2018, dussan2017, cintra2016, martinez2015, chaves2011, brander2013, aledo2009}.

In this paper, we are interested in the geometric Cauchy problem for \emph{constant-rank submanifolds}, that is, smooth embedded submanifolds of $\mathbb{R}^{m+c}$ whose Gauss map has constant rank, or, equivalently, whose second fundamental form has constant index of relative nullity~\cite{chern1952}. These constitute one of the simplest and most natural classes of submanifolds and, for this reason, also one of the most studied; see, e.g., \cite{raffaelli2023, onti2020, dajczer2017, florit2017, dajczer2013, vittone2012, dajczer2009, dajczer2004} and \cite[Chapter~7]{dajczer2019}.

\begin{problem}[Geometric Cauchy problem for constant-rank submanifolds]\label{pr:GCP}
Let $S$ be a connected $s$-dimensional smooth embedded submanifold of $\mathbb{R}^{m+c}$, and let $\mathcal{D} \supset TS$ be a smooth distribution of rank $m$ along $S$, that is, a smooth rank-$m$ subbundle of the ambient tangent bundle over $S$. Find all rank-$s$ submanifolds of $\mathbb{R}^{m+c}$ containing $S$ and whose tangent bundle along $S$ is precisely $\mathcal{D}$.
\end{problem}

Some special cases of Problem~\ref{pr:GCP} have already been considered: the case $m=2$ is classical \cite[pp.~198--200]{docarmo2016}; the case of rank-one hypersurfaces ($s=c=1$) was first addressed by Markina and the author in \cite{markina2019}; and the rank-one case in arbitrary codimension was solved by the author in \cite{raffaelli2024}. The interest in Problem~\ref{pr:GCP} is motivated by the search for new approaches, alternative to the Gauss parametrization~\cite{dajczer1985, sbrana1908}, to describe constant-rank submanifolds. The purpose of this paper is to present a solution of Problem~\ref{pr:GCP} without any assumption on $s$, $m$, or $c$.

To state our main result, let $\mathcal{D}^{\perp}$ be the distribution of rank $c$ whose fiber at $p \in S$ is the orthogonal complement of $\mathcal{D}_{p}$ in $\mathbb{R}^{m+c}$, let $\pi^{\top}$ denote orthogonal projection onto $\mathcal{D}$, and let $\phi_{p}$ be the map
\begin{align*}
T_{p}S \times \mathcal{D}_{p}^{\perp} &\to \mathcal{D}_{p}\\
(v,n) &\mapsto \pi^{\top} \nabla_{v}N, 
\end{align*}
where $\nabla$ is the Euclidean connection, and where $N$ is any extension of $n$ to a local section of $\mathcal{D}^{\perp}$. It is easy to see that $\phi_{p}$ is well defined.

\begin{theorem}\label{thm:GCPSolution}
Suppose that there exists a section $N^{\ast}$ of $\mathcal{D}^{\perp}$ such that the shape operator $A^{\ast} \coloneqq A_{N^{\ast}}$ of $S$ in direction $N^{\ast}$ is nonsingular, i.e., 
\begin{equation*}
\Rank A^{\ast} \rvert_{p} = s \quad \forall \+ p \in S.
\end{equation*}
The geometric Cauchy problem for rank-$s$ submanifolds of $\mathbb{R}^{m+c}$ has a solution if and only if $\phi_{p}(T_{p}S, N^{\ast}_{p}) = \Ima \phi_{p}$ for all $p \in S$. Moreover:
\begin{enumerate}[font=\upshape, label=(\roman*)]
\item\label{item1} The solution is locally unique, i.e., any two solutions coincide on an open set containing $S$.
\item\label{item2} In this open set, the unique solution is given by
\begin{equation*}
\bigl \{ p + x \mid p \in S, \+ x \in\mathcal{D}_{p} \cap ( \Ima \phi_{p})^{\perp}\bigr \},
\end{equation*}
where $^{\perp}$ denotes orthogonal complement in $\mathbb{R}^{m+c}$.
\item \label{item3} Given a local parametrization $\xi \colon A\to \mathbb{R}^{m+c}$ of $S$, the solution can be locally parametrized as follows. Let  $(E_{1}, \dotsc, E_{m})$ be a smooth orthonormal frame for $\mathcal{D}\rvert_{\xi(A)}$ whose first $s$ elements span $TS$. For any $j = 1, \dotsc, m-s$, let $G_{i}$ be the group of permutations of the set $\{ 1,\dotsc, \widehat{i},\dotsc,  s, s+j\}$ obtained by omitting the element $i$, and, denoting by $\phi_{i}^{k}$ the coordinate function $\langle \phi(E_{i}, N^{\ast}), E_{k} \rangle = \langle \nabla_{E_{i}} N^{\ast}, E_{k} \rangle$, let
\begin{equation}\label{eq:Xj}
X_{j} = \sum_{i\in \{1, \dotsc, s, s+j\}} (-1)^{h}\sum _{\lambda \in G_{i}} \Sign (\lambda) \phi_{1}^{\lambda(1)} \dotsm \phi_{s}^{\lambda(s+j)} E_{i},
\end{equation}
where $h =i$ for $i= 1, \dotsc, s$ and $h = s+1$ otherwise. Then
\begin{equation*}
\sigma(a, b^{1}, \dotsc, b^{m-s}) = \xi(a) + b^{1}X_{1}(a) + \dotsb + b^{m-s} X_{m-s}(a)
\end{equation*}
is the desired parametrization.
\end{enumerate}
\end{theorem}

\begin{remark}
$A^\ast = \pi_{S} \circ \phi (\cdot, N^\ast)$, where $\pi_S$ denotes orthogonal projection onto $TS$.
\end{remark}

\begin{remark}
The section $N^{\ast}$ does not need to be global. The theorem remains valid if one merely assumes the existence of a local section $N^{\ast}(p)$ in a neighborhood $U_{p}$ of each point $p\in S$, along with the compatibility condition
\begin{equation*}
N^{\ast}(p) = \pm N^{\ast}(q) \quad \forall \+ p,q \text{ such that }U_{p}\cap U_{q} \neq \emptyset.
\end{equation*}
\end{remark}

\begin{remark}
If the distribution $\mathcal{D}$ is obtained by restricting to $S$ the tangent bundle of an arbitrary submanifold $M \supset S$, then the map $\phi_p(\cdot, n)$ coincides with the restriction to $T_pS$ of the shape operator of $M$ with respect to $n$. In this case the theorem provides conditions for the existence of a constant-rank, first-order approximation of $M$ along $S$; cf.~\cite[section~6]{raffaelli2024}.
\end{remark}

\begin{remark}
When $s =1$, formula~\eqref{eq:Xj} reduces to $X_{j} = \phi^{1}E_{1+j} -\phi^{1+j}E_{1}$, and we thus retrieve the third item in \cite[Theorem~1.3]{raffaelli2024}. The formula is new starting from $s=2$, when it reduces to $X_{j} = -(\phi_{1}^{2}\phi_{2}^{2+j} - \phi_{1}^{2+j}\phi_{2}^{2})E_{1} + (\phi_{1}^{1}\phi_{2}^{2+j} - \phi_{1}^{2+j}\phi_{2}^{1})E_{2} - (\phi_{1}^{1}\phi_{2}^{2} - \phi_{1}^{2}\phi_{2}^{1})E_{j+2}$. 
\end{remark}

Clearly, when $c=1$, the condition on $\phi_{p}$ is automatically satisfied, and Problem~\ref{pr:GCP} becomes generically well posed.

\begin{corollary}\label{cor:GCPSolution}
Suppose that $c = 1$. If for each $p\in S$ the shape operator of $S$ in direction $n\in\mathcal{D}_p^{\perp}$ is nonsingular, then the geometric Cauchy problem for rank-$s$ hypersurfaces of $\mathbb{R}^{m+1}$ has a solution.
\end{corollary}

The proof of Theorem~\ref{thm:GCPSolution} will be given in section~\ref{sec:proof}. It relies on the well-known fact that any $m$-dimensional rank-$s$ submanifold of $\mathbb{R}^{m+c}$ admits a foliation by totally geodesic leaves of dimension $m-s$, along which the tangent space is constant; the converse holds provided that the foliation is unique. In particular, as shown in the next section, the tangent space is constant along the rulings precisely when the ambient covariant derivative of any tangent vector field along $S$ has vanishing normal component (Lemma~\ref{lm:criterion}). This extends a classical result about developable surfaces in $\mathbb{R}^{3}$~\cite[Theorem~54.1]{kreyszig1968}; see also \cite{yano1944}.


\section{Constant-rank submanifolds}

The purpose of this section is to generalize \cite[Lemma~3.7]{raffaelli2024} to \emph{$(m-s)$-ruled} submanifolds. 

\begin{definition}
Let $M$ be an $m$-dimensional embedded submanifold of $\mathbb{R}^{m+c}$. We say that $M$ is \textit{$(m-s)$-ruled} if it admits a foliation by $(m-s)$-dimensional totally geodesic submanifolds of $\mathbb{R}^{m+c}$, called \textit{rulings}. In particular, we say that $M$ is \textit{properly $(m-s)$-ruled} when the foliation is unique.
\end{definition}

Given a $s$-dimensional embedded submanifold $S$ of $\mathbb{R}^{m+c}$, the following lemma is the key to extending $S$ to an $(m-s)$-ruled submanifold. Recall that two embedded submanifolds $S,S'$ of a smooth manifold are said to \textit{intersect transversely} if for each $p\in S\cap S'$, the tangent spaces $T_{p}S$ and $T_{p}S'$ together span the ambient tangent space. 

\begin{lemma}\label{lm:construction}
Let $\mathcal{F}$ be a distribution of rank $m-s$ along $S$ such that
\begin{equation*}
\dim \mleft(T_{p}S \oplus \mathcal{F}_{p} \mright)= m \quad \forall \+ p \in S.
\end{equation*}
Then the subset
\begin{equation*}
\{ p+x \mid p\in S, \+ x \in \mathcal{F}_{p} \}
\end{equation*}
defines an $(m-s)$-ruled submanifold $M_{\mathcal{F}}$ in a neighborhood of $S$. Conversely, any $(m-s)$-ruled submanifold containing $S$ and whose rulings intersect $S$ transversely admits a description of this form.
\end{lemma}

\begin{proof}
We first prove a local version of the lemma. Let $\xi \colon A \to \mathbb{R}^{m+c}$ be a local parametrization of $S$, let $(X_{1}, \dotsc, X_{m-s})$ be a local frame for $\mathcal{F}$ over $A$, and, for $B \subset \mathbb{R}^{m-s}$ open, define $\sigma \colon A \times B \to \mathbb{R}^{m+c}$ by
\begin{equation*}
\sigma(a,b^{1}, \dotsc, b^{m-s}) = \xi(a) + b^{1} X_{1}(a) + \dotsb + b^{m-s} X_{m-s}(a).
\end{equation*}
Since the partial derivatives of $\sigma$ are linearly independent at $(a, 0)$, they remain linearly independent in some neighborhood $U_{a}$ of $(a,0)$ in $A \times B$. Moreover, after shrinking $U_{a}$ and $B$ if necessary, we may assume that $\sigma(U_{a})$ is embedded and there exists a neighborhood $V_{a}$ of $\xi(A)$ in $\mathbb{R}^{m+c}$ such that 
\begin{equation*}
\sigma(U_{a}) = V_{a} \cap \sigma(A \times B);
\end{equation*}
to see this, note that otherwise, by making $B$ arbitrarily close to $0$, the assumption that $S$ is embedded would be violated. Let $M\coloneqq \cup_{a \in A} U_{a}$, and let $q \in M$. Since $M$ is an immersed submanifold, the rank theorem~\cite[Theorem~4.12]{lee2013} implies the existence of a slice chart for $M$ in $\mathbb{R}^{m+c}$ around $q$. Hence $M$ satisfies the local slice criterion for embedded submanifolds~\cite[Theorem~5.8]{lee2013}, which is the desired local conclusion.

To prove the full statement, let $\{\xi_{i} \colon A_{i} \to \mathbb{R}^{m+c}\}_{i=1}^{r}$ be a set of local parametrizations of $S$ covering $S$, and let $\{\sigma_{i} \colon A_{i} \times B_{i} \to \mathbb{R}\}_{i=1}^{r}$ be the corresponding maps $\sigma$, as defined earlier. Then we know that for each $i$ there exists a neighborhood $U_{i}$ of $(A_{i},0)$ in $A_{i} \times B_{i}$ such that $\sigma_{i}(U_{i})$ is embedded. Shrinking $U_{1}, \dotsc, U_{r}$ along the second factor if necessary, we may assume that there are neighborhoods $V_{1},\dotsc, V_{r}$ of $\xi_{1}(A_{1}), \dotsc, \xi_{r}(A_{r})$ in $\mathbb{R}^{m+c}$ such that
\begin{equation*}
\sigma_{i}(U_{i}) = V_{i} \cap (\cup_{i=1}^{r} \sigma_{i}(U_{i})) \quad \forall \+ i =1, \dotsc, r.
\end{equation*}
As before, this implies the existence of a slice chart for $M_{i} =\sigma_{i}(U_{i})$ in $\mathbb{R}^{m+c}$ around every $q \in M_{i}$. Hence $M \coloneqq \cup_{i}M_{i}$ is the desired $(m-s)$-ruled submanifold.
\end{proof}

Next we define constant-rank submanifolds and give two alternative characterizations of them.

\begin{definition}
Let $M$ be an $m$-dimensional embedded submanifold of $\mathbb{R}^{m+c}$.
The \textit{index of relative nullity} of $M$ at a point $p \in M$ is the dimension of the kernel $\Delta$ of the second fundamental form $\alpha$ at $p$:
\begin{equation*}
\Delta = \{x \in T_{p}M \mid \alpha(x,\cdot) =0 \}.
\end{equation*}
We say that $M$ is a \textit{constant-rank} submanifold if the index of relative nullity is constant. In particular, we say that $M$ is \textit{rank-$s$} when the index is equal to $m-s$.
\end{definition}

\begin{theorem}[{\cite[Lemma~3.1]{hartman1965}}]
The following statements are equivalent:
\begin{enumerate}[font=\upshape]
\item $M$ is rank-$s$.
\item The Gauss map $M\to Gr(m,m+c)$ has rank $s$ everywhere; here $Gr(m,m+c)$ denotes the Grassmannian of $m$-dimensional subspaces of $\mathbb{R}^{m+c}$.
\item $M$ is properly $(m-s)$-ruled with constant tangent space along the $(m-s)$-rulings.
\end{enumerate}
\end{theorem}

Thus, rank-$s$ submanifolds containing $S$ are locally in one-to-one correspondence with distributions $\mathcal{F}$ of rank $m-s$ along $S$ making $M_{\mathcal F}$ properly ruled and satisfy
\begin{equation*}
 T_{p}M_{\mathcal{F}} = T_{p+x}M_{\mathcal{F}} \quad \forall \+ x \in \mathcal{F}_{p}\text{ such that } p+x \in M_{\mathcal{F}}.
\end{equation*}
The next lemma expresses the last condition in terms of the ambient covariant derivative along $S$ and represents the sought generalization of \cite[Lemma~3.7]{raffaelli2024}.

\begin{lemma}\label{lm:criterion}
The tangent space of $M_{\mathcal{F}}$ is constant along $x \in\mathcal{F}_{p}$ if and only if
\begin{equation*}
\nabla_{v}X \in T_{p}M_{\mathcal{F}} \quad \forall \+ v \in T_{p}S,
\end{equation*}
where $X$ is any extension of $x$ to a local section of $\mathcal{F}$.
\end{lemma}

\begin{proof}
Let $\xi \colon A \to \mathbb{R}^{m+c}$ be a local parametrization of $S$ around $p$, and let $(X_{j})_{j=1}^{m-s}$ be a local frame for $\mathcal{F}$ over $A$. Then the map $\sigma \colon A\times \mathbb{R}^{m-s} \to \mathbb{R}^{m+n}$ defined by
\begin{equation*}
\sigma(a,b^{1}, \dotsc, b^{m-s}) = \xi(a) + b^{1} X_{1}(a) + \dotsb + b^{m-s} X_{m-s}(a)
\end{equation*}
is a local parametrization of $M_{\mathcal{F}}$ in a neighborhood of $p$. Since the tangent space $T_{(a, 0, \dotsc, 0, b^{j}, 0, \dotsc, 0)} \sigma$ of $\sigma$ at $\sigma(a, 0, \dotsc, 0, b^{j}, 0, \dotsc, 0)$ is spanned by
\begin{equation*}
\partial_{1} \xi(a) + b^{j} \partial_{1} X_{j}(a), \dotsc, \partial_{s} \xi(a) + b^{j} \partial_{s} X_{j}(a) , X_{1}(a), \dotsc, X_{m-s}(a),
\end{equation*}
it follows that
\begin{equation*}
T_{(a, 0, \dotsc, 0, b^{j}, 0, \dotsc, 0)} \sigma = T_{(a,0)}\sigma = \Span \mleft( \partial_{1} \xi(a), \dotsc, \partial_{s} \xi(a), X_{1}(a), \dotsc, X_{m-s}(a)\mright)
\end{equation*}
if and only if $\partial_{i} X_{j}(a) \in T_{(a,0)}\sigma$ for every $i = 1, \dotsc, s$; by linearity, this is equivalent to all directional derivatives of $X_{j}$ being in $T_{(a,0)}\sigma$ at $a$. Now, to complete the proof, it is enough to note that any of the elements of the frame $(X_{j})_{j=1}^{m-s}$ may be chosen to coincide with $x$ at $p$.
\end{proof}

\begin{remark}
It is easy to see, using orthogonality of $T_{p}M_{\mathcal{F}}$ and $(T_{p}M_{\mathcal{F}})^{\perp}$, that the condition $\nabla_{v}X \in T_{p}M_{\mathcal{F}}$ in Lemma~\ref{lm:criterion} only depends on $X$ through $X_{p}=x \in T_{p}M_{\mathcal{F}}$. Hence the lemma continues to hold if ``local section of $\mathcal{F}$'' is replaced by ``local section of $TM_{\mathcal{F}}$''.
\end{remark}

\begin{remark}
By the properties of the covariant derivative, Lemma~\ref{lm:criterion} implies that the tangent space of $M_{\mathcal{F}}$ is constant along every $x \in \mathcal{F}$ if and only if 
\begin{equation*}
\nabla_{v_{i}}X_{j} \in T_{p}M_{\mathcal{F}} \quad \forall \+ i, j,
\end{equation*}
where $(v_{1}, \dotsc, v_{s})$ is a basis of $T_{p}S$ and $(X_{1}, \dotsc, X_{m-s})$ a local frame for $\mathcal{F}$ around $p$. 	
\end{remark}

\section{Proof of the main result}\label{sec:proof}
In this section we prove Theorem~\ref{thm:GCPSolution}. The proof of the first two items is based on the following lemma.

\begin{lemma}\label{lm:intersection}
Suppose that there exists a section $N^{\ast}$ of $\mathcal{D}^{\perp}$ such that the shape operator $A^{\ast}$ of $S$ in direction $N^{\ast}$ is nonsingular at $p$. Then
\begin{equation*}
T_{p}S \cap  \phi_{p}( T_{p}S, N^{\ast}_{p} )^{\perp} = \{0\}.
\end{equation*}
\end{lemma}
\begin{proof}
Let $\pi_{S}$ denote orthogonal projection onto $TS$, so that
\begin{equation*}
A_{N} = \pi_{S} \circ \phi (\cdot, N).
\end{equation*}
Suppose, towards a contradiction, that there exists a \emph{nonzero} vector $v \in T_{p}S$ such that 
\begin{equation*}
\langle v, \phi(w, N^{\ast}) \rangle = 0 \quad \forall \+ w \in T_{p}S.
\end{equation*}	
Since $\phi (\cdot, N)$ is self-adjoint with respect to the metric, this implies 
\begin{equation*}
A^{\ast}(v) = 0,
\end{equation*}	
contradicting the assumption that $A^{\ast}$ has full rank at $p$.
\end{proof}

\begin{proof}[Proof of Theorem~\textup{\ref{thm:GCPSolution}\ref{item1}--\ref{item2}}]
According to Lemmas \ref{lm:construction} and \ref{lm:criterion}, we need to find an $(m-s)$-dimensional subspace $\mathcal{F}_{p}$ of $\mathcal{D}_{p}$ satisfying $T_{p}S\cap \mathcal{F}_{p}= \{0\}$ and 
\begin{equation*}
\nabla_{v} x \in \mathcal{D}_{p} \quad \forall \+ (v,x) \in T_{p}S\times \mathcal{F}_{p} ,
\end{equation*}
where $x$ is being extended arbitrarily to a local section $X$ of $\mathcal{D}$. Note that the latter condition is equivalent to
\begin{equation}\label{eq:condition}
\mleft\langle \nabla_{v} x, n \mright\rangle =0 \quad \forall \+ (v,x,n) \in T_{p}S\times \mathcal{F}_{p} \times\mathcal{D}_{p}^{\perp};
\end{equation}
using orthogonality of $x$ and $n$, we can rewrite \eqref{eq:condition} as
\begin{equation*}
\mleft\langle x, \phi_{p}(v,n) \mright\rangle =0 \quad  \forall \+ (v,x,n) \in T_{p}S\times \mathcal{F}_{p} \times\mathcal{D}_{p}^{\perp}
\end{equation*}
or, more succinctly, as 
\begin{equation*}
\mathcal{F}_{p} \subset (\Ima \phi_{p})^{\perp}.
\end{equation*}

Suppose that there exists a section $N^{\ast}$ of $\mathcal{D}^{\perp}$ such that the shape operator $A^{\ast}$ of $S$ in direction $N^{\ast}$ is nonsingular at $p$, and set
\begin{equation*}
\mathcal{F}_{p}^{\ast} \coloneqq \mathcal{D}_{p}\cap\phi_{p}( T_{p}S, N^{\ast}_{p} )^{\perp}.
\end{equation*}
Then, since $A_{N^{\ast}} = \pi_{S} \circ \phi (\cdot, N^{\ast})$, we deduce that $\mathcal{F}_{p}^{\ast}$ is a subspace of dimension $m-s$; besides, it has trivial intersection with $T_{p}S$ by Lemma~\ref{lm:intersection}. Observing that 
\begin{equation*}
(\Ima \phi_{p})^{\perp} \subset \phi_{p}( T_{p}S, N^{\ast}_{p} )^{\perp},
\end{equation*}
we conclude that $\mathcal{F}^{\ast}$ can only define a solution of our problem when $\Ima \phi_{p} = \phi_{p}( T_{p}S, N^{\ast}_{p})$.

Next we check that $\mathcal{F}^{\ast}$ does indeed define a solution. To this end, let $M^{\ast} \coloneqq M_{\mathcal{F}^{\ast}}$ be an embedded submanifold containing $S$, whose existence is guaranteed by Lemma~\ref{lm:construction}. 
First note that, by construction, the shape operator $A^{\ast}$ of $S$ is nothing but the restriction to $TS$ of a shape operator of $M^{\ast}$. Hence, as the tangent space is constant along $\mathcal{F}^{\ast}$, the index of relative nullity of $M^{\ast}$ equals $m-s$ along $S$. By the same reason, extending $N^{\ast}$ via parallel translation in the ambient space to a unit normal vector field along $M^{\ast}$, we have that
\begin{equation*}
A^{\ast}_{p+x}(v) = A^{\ast}_{p}(v) \quad \forall \+ v \in \mathcal{D}_{p} \equiv T_{p+x}M^{\ast},
\end{equation*}
which implies that the index of relative nullity of $M^{\ast}$ is constant, as claimed.
\end{proof}

It remains to prove Theorem~\ref{thm:GCPSolution}\ref{item3}. To set the stage, we first recall the notion of ($(m-1)$-fold) vector cross product.

Let $V$ be an $m$-dimensional oriented inner product space, and let $(e_{1}, \dotsc, e_{m})$ be an orthonormal basis of $V$. A \textit{vector cross product on $V$} is a multilinear map $\cdot \times \dotsb \times \cdot \colon V^{m-1} \to V$ such that
\begin{align*}
&\langle e_{1}\times \dotsb \times e_{m-1}, e_{1} \rangle = \dotsb = \langle e_{1}\times \dotsb \times e_{m-1}, e_{m-1} \rangle =0,\\
&\langle e_{1}\times \dotsb \times e_{m-1}, e_{1}\times \dotsb \times e_{m-1} \rangle = \det(\langle e_{i}, e_{j} \rangle)_{i,j=1}^{m-1},\\
&(e_{1}, \dotsc, e_{m-1}, e_{1} \times\dotsb\times e_{m-1}) \text{ is positively oriented}.
\end{align*}
It is well known that there is only one vector cross product on $V$. It is given by
\begin{equation*}
v_{1}\times \dotsb \times v_{m-1}= \star(v_{1}\wedge \dotsb \wedge v_{m-1}),
\end{equation*}
where $\star$ denotes the Hodge star operator and $\wedge$ the wedge product on $V$; see \cite{brown1967}. More generally, if $v_{1}, \dotsc,v_{m-1}$ are continuous sections of a smooth oriented vector bundle over a smooth manifold, then both $\star$ and $\wedge$ are globally defined, and the previous equation defines their cross product.

Our proof of Theorem~\ref{thm:GCPSolution}\ref{item3} is based on the following lemma.

\begin{lemma}\label{lm:coim}
Let $W$ be an $s$-dimensional subspace of $V$, and let $\rho$ be a linear map $W \to V$; without loss of generality, we may assume that the first $s$ elements of $(e_{1}, \dotsc, e_{m})$ span $W$ and the first $r$ elements span $\coim\rho=(\ker\rho)^{\perp}$. Suppose that the orthogonal complement $\varphi(W)^{\perp}$ of $\varphi(W)$ in $V$ has trivial intersection with $\coim\rho$. Then 
\begin{equation*}
x_{j} = \varphi(e_{1}) \times \dotsb \times \varphi(e_{r})\times e_{r+1} \times \dotsb \times \widehat{e_{r+j}}\times \dotsb \times e_{m},  \quad j =1, \dotsc, m-r
\end{equation*}
defines an orthogonal basis of $\varphi(W)^{\perp}$; here the hat indicates that $e_{r+j}$ is omitted.
\end{lemma}
\begin{proof}
By construction, the vector $x_{j}$ belongs to the intersection between $\varphi(W)^{\perp}$ and the subspace spanned by $e_{1}, \dotsc, e_{r},e_{r+j}$. Since $\varphi(W)^{\perp}$ has trivial intersection with $\mathrm{span}(e_{1}, \dotsc, e_{r})$, it follows that $x_{1}, \dotsc, x_{m-r}$ all lie in orthogonal subspaces, showing that they indeed span $\varphi(W)^{\perp}$.
\end{proof}

\begin{proof}[Proof of Theorem~\textup{\ref{thm:GCPSolution}\ref{item3}}]
Let  $(E_{1}, \dotsc, E_{m})$ be a smooth orthonormal frame for $\mathcal{D}\rvert_{\xi(A)}$ whose first $s$ elements span $TS$. Equipping $\mathcal{D}\rvert_{\xi(A)}$ with the orientation naturally induced by that frame, the vector cross product becomes well defined on $\mathcal{D}\rvert_{\xi(A)}$. Then, for each $j = 1, \dotsc, m-s$, let 
\begin{equation}\label{eq:CrossProduct}
X_{j} = \phi(E_{1},N^{\ast}) \times \dotsb\times \phi(E_{s},N^{\ast}) \times E_{s+1} \times \dotsb \times \widehat{E_{s+j}} \times \dotsb \times E_{m},
\end{equation}
where the hat indicates that $E_{s+j}$ is omitted. Since $\coim \phi = TS$ by assumption, we deduce from Lemmas \ref{lm:intersection} and \ref{lm:coim} that $(X_{1}, \dotsc, X_{m-s})$ is an orthogonal frame for $\phi(TS, N^{\ast})^{\perp}$.

It remains to compute \eqref{eq:CrossProduct} in coordinates. Let $\phi_{i}^{k} = \langle \phi(E_{i}, N^{\ast}), E_{k})$. Then 
\begin{equation*}
\phi(E_{i},N^{\ast}) = \phi_{i}^{1}E_{1} + \dotsb + \phi_{i}^{m}E_{m}.
\end{equation*}
Substitution into \eqref{eq:CrossProduct} gives
\begin{align*}
X_{j} &= \bigl(\phi_{1}^{1}E_{1} + \dotsb + \phi_{1}^{s}E_{s}  + \phi_{1}^{s+j}E_{s+j}\bigr)\times \dotsb\times \left(\phi_{s}^{1}E_{1} + \dotsb + \phi_{s}^{s}E_{s}  + \phi_{s}^{s+j}E_{s+j}\right)\\ 
& \quad \times E_{s+1} \times \dotsb \times \widehat{E_{s+j}} \times \dotsb \times E_{m},	
\end{align*}
which implies \eqref{eq:Xj}.
\end{proof}

\bibliographystyle{amsplain}
\bibliography{biblio}
\end{document}